\documentclass[11pt]{amsart}
\usepackage{amsmath}
\DeclareMathOperator*{\argmin}{arg\,min}
\usepackage{amsfonts}

\usepackage{amscd}
\usepackage{amsthm}
\usepackage{amssymb} \usepackage{latexsym}
\usepackage{eufrak}
\usepackage{euscript}
\usepackage{epsfig}
\usepackage{graphics}
\usepackage{array}
\usepackage{enumerate}
\usepackage{dsfont}
\usepackage{color}
\usepackage{wasysym}
\usepackage{hyperref}
\usepackage{pdfsync}

\numberwithin{equation}{section}

\newcommand{\bel}[1]{\begin{equation}\label{#1}}

\newcommand{\be}{\begin{equation}}

\newcommand{\ba}{\begin{eqnarray}}
\newcommand{\ea}{\end{eqnarray}}

\newcommand{\qe}{\end{equation}}
\newcommand{\R}{{\mathbb R}}

\newcommand{\vol}{{\mathrm{vol}}}

\newcommand{\Hmm}[1]{\leavevmode{\marginpar{\tiny%
$\hbox to 0mm{\hspace*{-0.5mm}$\leftarrow$\hss}%
\vcenter{\vrule depth 0.1mm height 0.1mm width \the\marginparwidth}%
\hbox to
0mm{\hss$\rightarrow$\hspace*{-0.5mm}}$\\\relax\raggedright #1}}}

\theoremstyle{theorem}
\newtheorem{thm}{Theorem}[section]
\newtheorem{prop}{Proposition}[section]
\theoremstyle{example}
\newtheorem{example}{Example}[section]
\theoremstyle{corollary}
\newtheorem{coro}{Corollary}[section]
\theoremstyle{lemma}
\newtheorem{lemma}{Lemma}[section]
\theoremstyle{definition}
\newtheorem{defi}{Definition}[section]
\theoremstyle{proof}

\theoremstyle{remark}
\newtheorem{rem}{Remark}[section]

\begin{document}

\title[First eigenvalue estimates]{First eigenvalue estimates of Dirichlet-to-Neumann operators on graphs}

\author{Bobo Hua}
\address{School of Mathematical Sciences, LMNS, Fudan University, Shanghai 200433,
China; Shanghai Center for Mathematical Sciences, Fudan University, Shanghai 200433, China}
\email{bobohua@fudan.edu.cn}

\author{Yan Huang}
\address{School of Mathematical Sciences, Fudan University, Shanghai 200433, China.}
\email{yanhuang0509@gmail.com}

\author{Zuoqin Wang}
\address{School of Mathematical Sciences, University of Science and Technology of China, Hefei,
Anhui 230026, China.}
\email{wangzuoq@ustc.edu.cn}

\thanks{B.H. is supported in part by NSFC, NO.11401106. Z.W is supported in part by NSFC,  No. 11571131 and No. 11526212.}

\begin{abstract}
Following Escobar \cite{Escobar1997} and Jammes \cite{Jammes2015}, we introduce two types of isoperimetric constants and give lower bound estimates for the first nontrivial eigenvalues of Dirichlet-to-Neumann operators on finite graphs with boundary respectively.
\end{abstract}



\maketitle

\section{Introduction}
Let $(M,g)$ be a compact manifold with boundary $\partial{M}$. The Dirichlet-to-Neumann operator $\Lambda:H^{\frac{1}{2}}(\partial M)\longrightarrow H^{-\frac{1}{2}}(\partial M)$ is defined as
  $$\Lambda(f)=\frac{\partial u_f}{\partial n},$$
where $u_f$ is the harmonic extension of $f\in H^{\frac{1}{2}}(\partial M)$. The Dirichlet-to-Neumann operator is a first order elliptic pseudo-differential operator \cite[page 37]{Taylor1996} and its associated eigenvalue problem is also known as the Steklov problem, see \cite{Kuznetsov2014} for a historical discussion. Since $\partial M$ is compact, the spectrum of $\Lambda$ is nonnegative, discrete and unbounded \cite[page95]{Bandle1980}.

The Dirichlet-to-Neumann operator is closely related to the $\mathrm{Calder\acute{o}n}$
problem \cite{Calderon1980} of determining the anisotropic conductivity of a body from current and voltage measurements at its boundary. This makes it useful for applications to electrical impedance tomography, which is used in medical and geophysical imaging, see \cite{Uhlmann2014} for a recent survey.

Eigenvalue estimates are of interest in spectral geometry. In \cite{Cheeger1970}, Cheeger discovered a close relation between the geometric quantity, the isoperimetric constant (also called the Cheeger constant), and the analytic quantity, the first nontrivial eigenvalue of the Laplace-Beltrami operator on a closed manifold. Estimate of this type is called the Cheeger estimate.

For the first nontrivial eigenvalue of the Dirichlet-to-Neumann operator, two different types of lower bound estimates, which are similar to the classical Cheeger estimate, have been obtained by Escobar and Jammes respectively in \cite{Escobar1997} and \cite{Jammes2015}. For convenience, we call them the Escobar-type Cheeger estimate and the Jammes-type Cheeger estimate.

The Cheeger constant introduced by Escobar \cite{Escobar1997}, which we call the Escobar-type Cheeger constant, is defined as

$$h_E(M):=\inf_{\mathrm{Area}(\Omega\cap\partial M)\leq\frac{1}{2}\mathrm{Area}(\partial M)}\frac{\mathrm{Area}(\partial\Omega\cap\mathrm{int}(M))}{\mathrm{Area}(\Omega\cap\partial M)},$$
where $\mathrm{Area}(\cdot)$ denotes the codimensional one Hausdorff measure, i.e. the Riemannian area, and $\mathrm{int}(M)$ denotes the interior of the manifold $M$. Let $\sigma_1$ be the first nontrivial eigenvalue of the Dirichlet-to-Neumann operator $\Lambda$, then the Escobar-type Cheeger estimate \cite[Theorem 10]{Escobar1997} reads as
$$\sigma_1\geq\frac{\left(h_E(M)\mu_1(k)-ak\right)a}{a^2+\mu_1(k)},$$
where $a>0$, $k>0$ are arbitrary positive constants, and $\mu_1(k)$ is the first eigenvalue of the following Robin problem
\[
\left\{
       \begin{array}{ll}
       \Delta u+\mu_1(k)u=0 ,& \text{on\ }M, \\
       \frac{\partial u}{\partial n}+ku=0,& \text{on\ }\partial M.
               \end{array}
     \right.
\]

The Jammes-type Cheeger constant, which was introduced in \cite{Jammes2015}, is defined as
$$h_J(M)=\inf_{\mathrm{Vol}(\Omega)\leq\frac 12{\mathrm{Vol}(M)}}\frac{\mathrm{Area}(\partial\Omega\cap \mathrm{int}(M))}{\mathrm{Area}(\Omega\cap\partial M)},$$
where $\mathrm{Vol}(\cdot)$ denotes the Riemannian volume. The Jammes-type Cheeger estimate \cite[Theorem 1]{Jammes2015} is given as follows
$$\sigma_1\geq\frac{1}{4}h(M) h_J(M),$$
where $h(M)$ is the classical Cheeger constant associated to the Laplacian operator with Neumann boundary condition on $M$, which is defined as
$$h(M)=\inf_{\mathrm{Vol}(\Omega)\leq\frac 12{\mathrm{Vol}(M)}}\frac{\mathrm{Area}(\partial\Omega\cap \mathrm{int}(M))}{\mathrm{Vol}(\Omega)}.$$

The Dirichlet-to-Neumann operator is naturally defined in the discrete setting. We recall some basic definitions on graphs. Let $V$ be a countable set which serves as the set of vertices of a graph and
$\mu$ be a symmetric weight function given by
\[\aligned
\mu:V\times V & \to [0,\infty),\\
(x,y) &\mapsto \mu_{xy}=\mu_{yx}.
\endaligned\]
This induces a graph structure $G=(V,E)$ with the set of vertices $V$ and the set of edges $E$ which is defined as $\{x,y\}\in E$ if and only if $\mu_{xy}>0,$ in symbols $x\sim y.$ Note that we do allow self-loops in the graph, i.e. $x\sim x$ if $\mu_{xx}>0.$ Given $\Omega_1,\Omega_2\subset V$, the set of edges between $\Omega_1$ and $\Omega_2$ is denoted by
$$E(\Omega _1,\Omega _2):=\{e=\{x,y\}\in E|x\in\Omega _1,y\in\Omega _2\}.$$

For any subset $\Omega\subset V$, there are two notions of boundary, i.e. the edge boundary and the vertex boundary. The edge boundary of $\Omega$, denoted by $\partial\Omega$, is defined as
$$\partial\Omega:=E(\Omega,\Omega^c),$$
where $\Omega^c:=V\setminus\Omega$. The vertex boundary of $\Omega$, denoted by $\delta\Omega$, is defined as
$$\delta\Omega:=\{x\in V\setminus \Omega|\ x\sim y  \mathrm{\ for\ some\ } y\in \Omega\}.$$


Set $\overline{\Omega}:=\Omega\cup\delta\Omega$. We introduce a measure on $\overline{\Omega}$, $m:\overline{\Omega}\rightarrow(0,\infty)$, as follows
\[
m_x=\left\{
       \begin{array}{ll}
        \sum_{y\in V,y\sim x}\mu_{xy},& x\in\Omega, \\
        \sum_{y\in \Omega, y\sim x}\mu_{xy},&x\in\delta\Omega.
        \end{array}
     \right.
\]
Accordingly, $m(A):=\sum_{x\in A}m_x$ denotes the measure of any subset $A\subset\overline\Omega$. Given any set $F,$ we denote by $\R^F$ the collection of all real functions defined on $F.$

For any finite subset $\Omega\subset V$, in analogy to the Riemannian case, one can define the Dirichlet-to-Neumann operator in the discrete setting to be
\[\aligned
\Lambda:\R^{\delta\Omega} & \to\R^{\delta\Omega},
\\
\varphi  &\mapsto \Lambda\varphi:=\frac{\partial u_\varphi}{\partial n},
\endaligned\]
where $u_\varphi$ is the harmonic extension of $\varphi$ to $\Omega$, and $\frac \partial{\partial n}$ is the outward normal derivative in the discrete setting defined as in \eqref{neumannboundaryoperator} in section 2. We call $\Lambda$ the DtN operator for short. Let $\sigma(\Lambda)$ denote the spectrum of $\Lambda$. By the definition of $\Lambda$ and Green's formula, see Lemma \ref{Green's formula}, $\Lambda$ is a nonnegative self-adjoint operator, i.e. $\sigma(\Lambda)$ is a set of nonnegative real numbers. In fact, $\sigma(\Lambda)$ is contained in $[0,1]$, see Proposition \ref{eigenvaluebound}.

It is well known that the eigenvalues of (normalized) Laplace operators on finite graphs without any boundary condition are contained in $[0,2]$. The multiplicity of eigenvalue 0 is equal to the number of connected components of the graph, see \cite[Prop.1.3.7]{BrouwerHaemers2012}. The largest possible eigenvalue 2 is achieved if and only if the graph is bipartite, see \cite[Theorem 2.3]{Grigor'yan2009}. Similar results can be generalized to the case of DtN operators. For convenience, let $\widetilde{\Omega}$ denote the graph with vertices in $\overline{\Omega}$ and edges in $E(\Omega,\overline{\Omega})$, i.e. \begin{eqnarray}\label{graphinfact}\widetilde{\Omega}:=(\overline{\Omega},E(\Omega,\overline{\Omega})).\end{eqnarray}
From Proposition \ref{0eigenvalue}, we know that the multiplicity of eigenvalue 0 of $\Lambda$ is equal to the number of connected components of $\widetilde{\Omega}$. Moreover, we show that the eigenspace associated to the eigenvalue 1 which is the largest possible eigenvalue is the kernel of a linear operator $Q$ whose definition is given in \eqref{qoperator} below.



Given $A\subset\overline{\Omega}$, we denote by
\[\partial_{\Omega}A:=\partial A\cap E(\Omega,\overline{\Omega})\]
the relative boundary and $A^{\vee}:=\overline{\Omega}\setminus A$ the relative complement of $A$ in $\widetilde{\Omega}$. Without loss of generality, we always assume that $\widetilde{\Omega}$ is connected.

Now we consider the first nontrivial eigenvalue estimates of DtN operators on finite graphs. We introduce two isoperimetric constants following Escobar and Jammes, which we call Escobar-type Cheeger constant and Jammes-type Cheeger constant respectively.

\begin{defi}
The Escobar-type Cheeger constant for $\widetilde{\Omega}$ is defined as
$$h_E(\widetilde{\Omega}):=\inf_{m(A\cap\delta\Omega)\leq\frac12m(\delta\Omega)}\frac{\mu(\partial _{\Omega}A)}{m(A\cap\delta\Omega)}.$$
The Jammes-type Cheeger constant for $\widetilde{\Omega}$ is defined as
$$h_J(\widetilde{\Omega}):=\inf_{m(A)\leq \frac 12{m(\overline{\Omega})}}\frac{\mu(\partial _{\Omega}A)}{m(A\cap\delta\Omega)}.$$
\end{defi}
By definitions, one has $h_J(\widetilde{\Omega})\leq h_E(\widetilde{\Omega})$, see Proposition \ref{jammesescobarcompare}. This estimate is optimal, see Example \ref{exampleforjammesescobarcompare} in the paper.

To derive the Cheeger estimates, we first show that $h_E(\widetilde{\Omega})$ is equal to a type of Sobolev constant. Similar results can be found in both Riemannian and discrete case, see e.g. \cite{GeometryAnalysis2012,Chang2014,Chung97}.

\begin{thm}\label{sobolevconstant}
Let $G$ be a finite graph and $\Omega\subset V,$ then
$$h_E(\widetilde{\Omega})=\inf_{f\in\R^{\overline\Omega}}\frac{\sum_{e=\{x,y\}\in E(\Omega,\overline\Omega)}\mu_{xy}|f(x)-f(y)|}{\inf_{a\in\R}\sum_{x\in\delta\Omega}m_x|f(x)-a|}.$$
\end{thm}

Without loss of generality, we assume that the number of vertices in $\delta\Omega$ is at least two. From now on, we denote by $\lambda_1(\Omega)$ the first nontrivial eigenvalue of the DtN operator $\Lambda$ on $\Omega$ in a finite graph. In the discrete setting, it's easy to obtain an upper bound estimate as
$$\lambda_1(\Omega)\leq 2h_E(\widetilde{\Omega}),$$
see Proposition \ref{upperbound}. The sharpness of this upper bound can be seen from Example \ref{exampleforupperbound}. For the lower bound estimate, we obtain the Escobar-type Cheeger estimate as follows.
\begin{thm}\label{escobarlowerbound}
Let $G$ be a finite graph and $\Omega\subset V,$ then
\begin{eqnarray}\label{escobarcheegerestimate}\lambda_1(\Omega)\geq\frac{\left(2h_E(\widetilde{\Omega})\mu_1(k)-a(k+\mu_1(k))\right)a}{a^2+2\mu_1(k)},\end{eqnarray}
where $a>0$, $k>0$ are arbitrary positive constants, and $\mu_1(k)$ is the first eigenvalue of the Robin problem
\[
\left\{
       \begin{array}{ll}
       \Delta u+\mu_1(k)u=0 ,& \text{on\ }\Omega, \\
       \frac{\partial u}{\partial n}+ku=0,& \text{on\ }\partial \Omega.
               \end{array}
     \right.
\]
\end{thm}

Analogous to the Riemannian case, we obtain Jammes-type Cheeger estimate following \cite{Jammes2015}.
\begin{thm}\label{thm:main3}
Let $G$ be a finite graph and $\Omega\subset V,$ then
$$\lambda_1(\Omega)\geq \frac{1}{2}h(\widetilde{\Omega})h_J(\widetilde{\Omega}),$$
where $h(\widetilde{\Omega})$ is the classical Cheeger constant of the graph $\widetilde{\Omega}$ viewed as a graph without any boundary condition. 
\end{thm}

\begin{rem}
The Jammes-type Cheeger estimate is asymptotically sharp, see Example \ref{exapmle}.
\end{rem}

For the completeness, we recall the definition of $h(\widetilde{\Omega})$,
$$h(\widetilde{\Omega}):=\inf_{m(A)\leq\frac{1}{2}m(\overline\Omega)}\frac{\mu(\partial_{\Omega}A)}{m(A)},$$
where $A$ is any nonempty subset of $\overline\Omega$, see e.g. \cite[p.24]{Chung97}. By the definitions of $h(\widetilde{\Omega})$ and $h_J(\widetilde{\Omega})$, one is ready to see that $h(\widetilde{\Omega})\leq h_J(\widetilde{\Omega})$. Notably, the classical Cheeger estimate uses only one Cheeger constant while the Jammes-type Cheeger estimate involves two. One may ask whether $\lambda_1(\Omega)$ can be bounded from below using only $h_J(\widetilde{\Omega})$. The answer is negative and we give a counterexample in Example \ref{exapmle}.

As a corollary of Theorem \ref{thm:main3}, we have the following interesting eigenvalue estimate, which has no counterpart in the Riemannian setting.
\begin{coro}\label{application}
$$\lambda_1(\Omega)\geq\frac{1}{8}(\zeta_1(\widetilde{\Omega}))^2,$$
where $\zeta_1(\widetilde{\Omega})$ is the first nontrivial eigenvalue of the Laplace operator with no boundary condition on $\widetilde{\Omega}$.
\end{coro}

The organization of the paper is as follows: In Section 2, we collect some basic facts about the DtN operators on graphs. In Section 3, we study the spectrum of the DtN operators. In Section 4 and Section 5, we give the proof of the main theorems: Theorem \ref{escobarlowerbound} and \ref{thm:main3}.

\section{preliminaries}

Let $(X,\nu)$ be a discrete measure space, i.e. $X$ is a discrete space equipped with a Borel measure $\nu.$ The spaces of $\ell^p,$ $p\in[1,\infty]$, summable functions on $(X,\nu)$, are defined routinely: Given a function $f\in\R^X,$ for $p\in[1,\infty)$, we denote by
$$\|f\|_{\ell^p}=\left(\sum_{x\in X}|f(x)|^p\nu(x)\right)^{1/p},$$
the $\ell^p$ norm of $f.$ For $p=\infty,$
$$\|f\|_{\ell^\infty}=\sup_{x\in X}|f(x)|.$$
Let $\ell^p(X,\nu):=\{f\in\R^X| \|f\|_{\ell^p}<\infty\}$ be the space of $\ell^p$ summable functions on $(X,\nu).$ In our setting, these definitions apply to $(X,\nu)=(\Omega,m)$ or $(\delta\Omega,m).$ The case where $p=2$ is of particular interest, as we have the Hilbert spaces $\ell^2(\Omega,m):=\{f\in\R^\Omega|\|f\|_{\ell^2}<\infty\}$ and $\ell^2(\delta\Omega,m):=\{\varphi\in\R^{\delta\Omega}|\|\varphi\|_{\ell^2}<\infty\}$ equipped with the standard inner products
\[\aligned
\langle f,g\rangle_{\Omega} & =\sum_{x\in \Omega}f(x)g(x)m_x,\quad f,g:\Omega\to\R,
\\
\langle \varphi,\psi\rangle_{\delta\Omega} &=\sum_{x\in\delta\Omega}\varphi(x)\psi(x)m_x,\quad \varphi,\psi:\delta\Omega\to\R.
\endaligned\]


Given $\Omega\subset V,$ an associated quadratic form is defined as
$$D_\Omega(f,g)=\sum_{e=\{x,y\}\in E(\Omega,\overline{\Omega})}\mu_{xy}(f(x)-f(y))(g(x)-g(y)), \quad f, g\in\R^{\overline{\Omega}}.$$
The Dirichlet energy of $f\in\R^{\overline{\Omega}}$ can be written as
\[D_\Omega(f):=D_\Omega(f,f).\]
For any $f\in\R^{\overline{\Omega}},$ the Laplacian operator is defined as
$$\Delta f(x):=\frac{1}{m_x}\sum_{y\in V: y\sim x}\mu_{xy}(f(y)-f(x)),\quad x\in \Omega$$
and the outward normal derivative of $f$ is defined as
\begin{equation}\label{neumannboundaryoperator}\frac{\partial f}{\partial n}(z):=\frac{1}{m_z}\sum_{x\in\Omega: x\sim z}\mu_{zx}(f(z)-f(x)),\quad z\in\delta\Omega.\end{equation}

We recall the following two well-known results on Laplace operators.
\begin{lemma}\label{Green's formula}
(Green's formula) Let $f,g\in \mathbb R^{\overline{\Omega}}.$ Then
\begin{equation}\label{eq:green formula}\langle \Delta f,g\rangle_\Omega=-D_{\Omega}(f,g)+\langle \frac{\partial f}{\partial n},g\rangle_{\delta\Omega}.\end{equation}
\end{lemma}

\begin{lemma}\label{harmonic extension}
  Given any $\varphi\in\R^{\delta\Omega},$ there is a unique function $u_{\varphi}\in\R^{\overline{\Omega}}$ satisfying the Laplace equation
  \begin{equation}\label{eq:eq1}\Delta u_\varphi(x)=0,\quad x\in\Omega,\end{equation}
and the boundary condition
\begin{equation*}u_\varphi(x)=\varphi(x),\quad x\in\delta\Omega.\end{equation*}
\end{lemma}
\begin{rem} We will denote by $u_{\varphi}$ the harmonic extension of $\varphi\in\mathbb{R}^{\delta\Omega}$ to $\overline{\Omega}$ in this paper. For the proof of Lemma \ref{Green's formula} and \ref{harmonic extension}, one can see e.g. \cite{Grigor'yan2009}.
\end{rem}
\begin{prop}
$\Lambda$ is a nonnegative self-adjoint linear operator on $\ell^2(\delta\Omega,m)$.
\end{prop}
\begin{proof}
For any $f$, $g\in\ell^2(\delta\Omega)$, by Green's formula, we have
$$0=\langle \Delta u_f,u_g\rangle_\Omega=-D_{\Omega}(u_f,u_g)+\langle \frac{\partial f}{\partial n},g\rangle_{\delta\Omega},$$
$$0=\langle  u_f,\Delta u_g\rangle_\Omega=-D_{\Omega}(u_f,u_g)+\langle f,\frac{\partial g}{\partial n}\rangle_{\delta\Omega}.$$
Hence
$$\langle\Lambda f,g\rangle=\langle f,\Lambda g\rangle.$$
In particular,
$$0=\langle \Delta u_f,u_f\rangle_\Omega=-D_{\Omega}(u_f,u_f)+\langle \frac{\partial f}{\partial n},f\rangle_{\delta\Omega},$$
then
$$\langle\Lambda f,f\rangle_{\delta\Omega}=D_\Omega(u_f,u_f)\geq 0.$$
So we complete the proof.
\end{proof}

For any $y\in \delta\Omega,$ let $\delta_{y}(z)$ denote the delta function at $y,$ i.e. $\delta_y(y)=1$ and $\delta_y(z)=0$ for any $z\in\delta\Omega, z\neq y.$  We denote by $P(\cdot,y)$ the solution of equation \eqref{eq:eq1} with Dirichlet boundary condition $\frac{1}{m_y}\delta_y.$ The family $\{P(\cdot,y)\}_{y\in\delta\Omega}$ are called Poission kernels associated to the Dirchlet boundary conditions on $\Omega.$ They have interesting probabilistic explanations using simple random walks, see e.g. \cite[p.25]{Lawler2010} and \cite[chapter 8]{LawlerLimic2010}. Using Poission kernels, one can regard the DtN operator on $\Omega$ as a Laplace operator defined on a graph with the set of vertices  $\delta\Omega$ and modified edges.
\begin{prop}\label{equimap}
The DtN operator can be written as
$$\Lambda\varphi(x)=\frac{1}{m_x}\sum_{y\in\delta\Omega}\widetilde{\mu}_{xy}(\varphi(x)-\varphi(y)),$$
where $\widetilde{\mu}_{xy}=\sum_{z\in\Omega}\mu_{xz}P(z,y)m_y$ and $\sum_{y\in\delta\Omega}\widetilde{\mu}_{xy}=m_x$.
\end{prop}
\begin{proof}
For any given $\varphi\in \mathbb{R}^{\delta\Omega}$, we have
$$\varphi(x)=\sum_{y\in\delta\Omega}\varphi(y)m_y\cdot\frac{\delta_y(x)}{m_y}.$$
By the linearity of equation \eqref{eq:eq1}, we have
\begin{eqnarray}\label{poissonkernel}
u_{\varphi}(x)=\sum_{y\in\delta\Omega}\varphi(y)m_yP(x,y).
\end{eqnarray}
Hence by the definition of $\Lambda$,
\begin{eqnarray*}
\Lambda\varphi(x)&=&\frac{1}{m_x}\sum_{z\in\Omega}\mu_{xz}\left(\varphi(x)-\sum_{y\in\delta\Omega}\varphi(y)
m_yP(z,y)\right)\nonumber\\
&=&\varphi(x)-\frac{1}{m_x}\sum_{y\in\delta\Omega}\left(\sum_{z\in\Omega}\mu_{xz}P(z,y)m_y\right)\varphi(y)\nonumber\\
&=&\varphi(x)-\frac{1}{m_x}\sum_{y\in\delta\Omega}\widetilde{\mu}_{xy}\varphi(y).
\end{eqnarray*}
Notice that $u_{\varphi}(x)\equiv1$, if we choose $\varphi(x)=1,$ $x\in\delta\Omega$. Hence combining with  \eqref{poissonkernel}, we have
\begin{eqnarray*}
\sum_{y\in\delta\Omega}\widetilde{\mu}_{xy}&=&\sum_{z\in\Omega}\mu_{xz}\left(\sum_{y\in\delta_e\Omega}P(z,y)m_y\right)\nonumber\\
&=&\sum_{z\in\Omega}\mu_{xz}\cdot 1=m_x.
\end{eqnarray*}
Then the proposition follows.
\end{proof}

From Proposition \ref{equimap}, the DtN operator $\Lambda$ can be written in a matrix form. We define matrices $D_{\delta\Omega},$ $A_{\delta\Omega\times\Omega},$ $P_{\Omega\times\delta\Omega}$ as
\[\aligned
(D_{\delta\Omega})_{xy}& =m_x\delta_x(y),\quad   x,y\in\delta\Omega,\\
(A_{\delta\Omega\times\Omega})_{xz} & =\mu_{xz},\quad\qquad x\in\delta\Omega, z\in \Omega, \\
(P_{\Omega\times\delta\Omega})_{zx}& =P(z,x),\;\quad x\in\delta\Omega, z\in \Omega.
\endaligned\]
Then we have
\begin{coro}\label{matrixofDTN}
The DtN operator $\Lambda$ can be written as
$$\Lambda=I-D^{-1}APD,$$ where $I$ is the identity map.
\end{coro}

\section{Spectrum of the DtN operator}

Given $\varphi\in\R^{\delta\Omega}$, for simplicity we denote by $\overline{\varphi}$ the null-extension of $\varphi$ to $\overline{\Omega}$, which is defined as
\[
\overline{\varphi}(x)=\left\{
       \begin{array}{ll}
        0,& x\in\Omega, \\
        \varphi(x),&x\in\delta\Omega.
               \end{array}
     \right.
\]

For any $p\in[1,\infty],$   the $\ell^p$-$\ell^p$ norm of the operator $\Lambda$ is defined as
$$\|\Lambda\|_{p,p}:=\sup_{\varphi\in\R^{\delta\Omega},\|\varphi\|_{\ell^p}=1}\|\Lambda \varphi\|_{\ell^p}.$$

\begin{prop}\label{eigenvaluebound}
The $\ell^2$-$\ell^2$ and $\ell^\infty$-$\ell^\infty$ norm of the operator $\Lambda$ are bounded, in particular
  $$\|\Lambda\|_{2,2}\leq 1,\quad \|\Lambda\|_{\infty,\infty}\leq 2.$$
\end{prop}
\begin{proof}
Let $\varphi\in\R^{\delta\Omega}.$  According to H\"older's inequality,
\begin{eqnarray}\label{eq:eq2}
\|\Lambda\varphi\|_{\ell^2(\delta\Omega,m)}^2&=&\sum_{x\in\delta\Omega}m_x\left|\sum_{y\in\Omega}
\frac{\mu_{xy}}{m_x}(\varphi(x)-u_{\varphi}(y))\right|^2\nonumber\\
&\leq&\sum_{x\in\delta\Omega}\sum_{y\in\Omega}\mu_{xy}|\varphi(x)-u_{\varphi}(y)|^2\leq D_{\Omega}(u_{\varphi})\nonumber\\
&\leq&D_{\Omega}(\overline{\varphi})=\sum_{x\in\delta\Omega,y\in\Omega}\mu_{xy}\varphi^2(x)\\
&=&\|\varphi\|_{\ell^2(\delta\Omega,m)}^2,\nonumber\end{eqnarray} where \eqref{eq:eq2} follows from the fact that harmonic functions minimize the Dirichlet energy in the class of functions with the same boundary condition. This proves the  $\ell^2$-$\ell^2$ norm bound.

For the  $\ell^\infty$-$\ell^\infty$ norm bound, we have
\begin{eqnarray*}
\parallel\Lambda\varphi\parallel_{\infty}&=&\sup_{x\in\delta\Omega}\left|\frac{1}{m_x}\sum_{y\in\Omega}
\mu_{xy}(\varphi(x)-u_{\varphi}(y))\right|\nonumber\\
&\leq&\sup_{x\in\delta\Omega}\frac{1}{m_x}\sum_{y\in\Omega}\mu_{xy}|(\varphi(x)-u_{\varphi}(y))|\nonumber\\
&\leq&2\parallel\varphi\parallel_{\infty}.
\end{eqnarray*}
\end{proof}
\begin{rem}
From Proposition \ref{eigenvaluebound}, we have $\sigma(\Lambda)\subset [0,1]$. \end{rem}
By interpolation, we have
\begin{coro}
  $\|\Lambda\|_{p,p}\leq 2^{1-\theta}$, where $\theta\in[0,1)$ and $p=\frac{2}{\theta}\in[2,\infty].$
\end{coro}
\begin{prop}\label{0eigenvalue}
The multiplicity of the eigenvalue 0 of $\Lambda,$ i.e. $\dim \mathrm{Ker}\Lambda,$ is equal to the number of connected components of the graph $\widetilde{\Omega}.$
\end{prop}
\begin{proof}
Let $\varphi\in\R^{\delta\Omega}$ be an eigenfunction associated to eigenvalue 0 of $\Lambda,$ i.e. $\Lambda\varphi=0.$ By Green's formula we have
$$0=\langle \Delta u_{\varphi},u_{\varphi}\rangle_\Omega=-D_{\Omega}(u_{\varphi})+\langle \Lambda\varphi,\varphi\rangle_{\delta\Omega}=-D_{\Omega}(u_{\varphi}).$$
Hence $u_{\varphi}$ is constant on each connected component of $\widetilde{\Omega}.$
\end{proof}
From now on, we always assume that the graph $\widetilde{\Omega}$ is connected and $\Lambda$ is an operator from $\ell^2(\delta\Omega,m)$ to $\ell^2(\delta\Omega,m).$ Let \[E_1(\Lambda):=\{\varphi\in\R^{\delta\Omega}|\ \Lambda\varphi=\varphi\}\]
be the space of eigenvectors associated to the eigenvalue $1$ which might be empty. Set $\delta_I\Omega:=\{x\in \Omega|\ x\sim y\ \mathrm{for\ some\ }y\in \delta\Omega\}.$ For any $\varphi\in\R^{\delta\Omega},$ we introduce a linear operator $Q:\R^{\delta\Omega}\rightarrow\R^{\delta_I\Omega},$ which is defined as
\begin{eqnarray}\label{qoperator}Q\varphi(x)=\frac{1}{m_x}\sum_{y\in\delta\Omega}\mu_{xy}\varphi(y),\quad x\in\delta_I\Omega.\end{eqnarray}

Let $\sharp\delta\Omega$ ($\sharp\delta_I\Omega$ resp.) denote the number of vertices in $\delta\Omega$ ($\delta_I\Omega$ resp.). We order the eigenvalues of the DtN operator $\Lambda$ in the nondecreasing way:
$$0=\lambda_0(\Omega)<\lambda_1(\Omega)\leq\cdots\leq \lambda_{N-1}(\Omega)\leq 1,$$ where $N=\sharp\delta\Omega.$

We obtain some characterisations of $E_1(\Lambda)$ in the following proposition.

\begin{prop}\label{thm:main2}
(1) For $\varphi\in\R^{\delta\Omega},$ $\varphi\in E_1(\Lambda)$ if and only if
$$u_\varphi=\overline{\varphi}.$$

(2) $$E_1(\Lambda)=\mathrm{Ker}Q.$$
\end{prop}
\begin{proof}(1)
For the "if " part,
\[\Lambda\varphi(x)=\frac{1}{m_x}\sum_{y\in\Omega}\mu_{xy}(\overline{\varphi}(x)-\overline{\varphi}(y))=\varphi(x).\]

For the "only if " part, if $\varphi\in E_1(\Lambda),$ then all the inequalities in \eqref{eq:eq2} are equalities. Hence $D_{\Omega}(u_\varphi)=D_{\Omega}(\overline{\varphi}).$ This implies that $\overline{\varphi}$ attains the minimal Dirichlet energy with fixed boundary condition, i.e. $\overline{\varphi}$ is harmonic. By the uniqueness of harmonic functions with fixed boundary condition, we have $\overline{\varphi}=u_{\varphi}.$

(2) If $\varphi\in E_1(\Lambda)$, then $u_{\varphi}=\overline{\varphi}.$ For any $x\in\delta_I\Omega$,
\begin{eqnarray*}
Q\varphi(x)&=&\frac{1}{m_x}\sum_{y\in\delta\Omega}\mu_{xy}\varphi(y)\nonumber\\
&=&\frac{1}{m_x}\left(\sum_{y\in\delta\Omega}\mu_{xy}(\overline{\varphi}(y)-\overline{\varphi}(x))
+\sum_{y\in\Omega}\mu_{xy}(\overline{\varphi}(y)-\overline{\varphi}(x))\right)\nonumber\\
&=&\Delta u_{\varphi}(x)=0.
\end{eqnarray*}
Hence $E_1(\Lambda)\subset \mathrm{Ker}Q$. On the other hand, If $\varphi\in \mathrm{Ker}Q$, then for any $x\in\delta_I\Omega$,
$$\Delta\overline{\varphi}(x)=\frac{1}{m_x}\sum_{y\in\delta\Omega}\mu_{xy}(\overline{\varphi}(y)-
\overline{\varphi}(x))=\frac{1}{m_x}\sum_{y\in\delta\Omega}\mu_{xy}\varphi(y)=0.$$
Hence $u_{\varphi}=\overline{\varphi}$, i.e. $\mathrm{Ker}Q\subset E_1(\Lambda)$, and the proof is completed.

\end{proof}

\begin{rem}
By Proposition \ref{thm:main2}, the problem of determining $E_1(\Lambda)$ can be reduced to the properties of the combinatorial structure of $\delta_I\Omega\cup\delta\Omega$, which is independent of the inner structure $\Omega\setminus\delta_I\Omega$.
\end{rem}

From Proposition \ref{thm:main2}, we obtain a sufficient condition for $E_1(\Lambda)$ to be nonempty as follows.
\begin{coro}\label{eigenvalue1}
  $$\dim E_1(\Lambda)\geq \sharp\delta\Omega-\sharp\delta_I\Omega.$$
In particular, $E_1(\Lambda)\neq\emptyset$ if $\sharp\delta\Omega>\sharp\delta_I\Omega.$
\end{coro}
\begin{proof}
It directly follows from the fact that $\dim{\mathrm{Ker}Q}+\dim{\mathrm{Im}Q}=\sharp\delta\Omega$.
\end{proof}

\section{escobar-type Cheeger estimate}
\begin{prop}\label{jammesescobarcompare}
Let $G$ be a finite graph and $\Omega\subset V,$ then we have
$$h_J(\widetilde{\Omega})\leq h_E(\widetilde{\Omega}).$$
\end{prop}
\begin{proof}
Choose $A\subset\overline{\Omega}$ that achieves $h_E(\widetilde{\Omega})$, i.e.
\[m(A\cap\delta\Omega)\leq\frac{m(\delta\Omega)}{2}\quad \text{and}\quad h_E(\widetilde{\Omega})=\frac{\mu(\partial_{\Omega}A)}{m(A\cap\delta\Omega)}.\]

If $m(A)\leq\frac{m(\overline{\Omega})}{2}$, then $h_J(\widetilde{\Omega})\leq h_E(\widetilde{\Omega}).$
If $m(A)\geq\frac{m(\overline{\Omega})}{2}$, i.e. $m(A^c)\leq\frac{m(\overline{\Omega})}{2}$, then
$$\frac{\mu(\partial_{\Omega}A^c)}{m(A^c\cap\delta\Omega)}=\frac{\mu(\partial_{\Omega}A)}
{m(A^c\cap\delta\Omega)}\leq\frac{\mu(\partial_{\Omega}A)}{m(A\cap\delta\Omega)}.$$
Hence in both cases we have $h_J(\widetilde{\Omega})\leq h_E(\widetilde{\Omega}).$
\end{proof}

From the following example, we know that the equality in the above proposition can be achieved.

\begin{example}\label{exampleforjammesescobarcompare}
Consider the path graph $P_6$ as shown in Figure 1 with unit edge weights, $\Omega=\{v_2,v_3,v_4,v_5\}$ and $\delta\Omega:=\{v_1,v_6\}$. By computation, we have $h_E(\widetilde{\Omega})=h_J(\widetilde{\Omega})=1.$
\end{example}

\begin{figure}[!h]
\includegraphics[height=4cm,width=9cm]{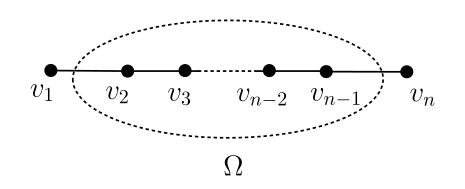}
\caption{}\label{figure1}
\end{figure}

For convenience, we need the following notion.
\begin{defi}
For any $f\in\mathbb{R}^{\delta\Omega}$, we call $k\in\R$ the $L^1$-mean value of $f$ over $\delta\Omega$ if $k$ satisfies
$$m(\{x\in\delta\Omega|f(x)\geq k\})\geq\frac{1}{2}m(\delta\Omega)$$
and
$$m(\{x\in\delta\Omega|f(x)\leq k\})\geq\frac{1}{2}m(\delta\Omega).$$
\end{defi}

\begin{rem}
The $L^1$-mean value may not be unique in general. For simplicity, we denote by $\overline{f}$ the $L^1$-mean value of $f\in\R^{\delta\Omega}$.
\end{rem}

\begin{lemma}\label{eqsobolev}
\begin{displaymath}
\argmin_{k\in\R}\sum_{x\in\delta\Omega}m_x|f(x)-k|=\overline{f},
\end{displaymath}
where $\argmin$ denotes the value $k$ at which $\sum_{x\in\delta\Omega}m_x|f(x)-k|$ attains the minimum.
\end{lemma}

\begin{rem}\label{equivalence}
(a) For the proof of Lemma \ref{eqsobolev}, one refers to e.g. \cite{ChangShaoZhang2015}.

(b) From Lemma \ref{eqsobolev}, Theorem \ref{sobolevconstant} is equivalent to
\begin{eqnarray}\label{sobolevtypecharactrisation}
h_E(\widetilde{\Omega})=\inf_{f\in\R^{\overline\Omega}}\frac{\sum_{e=\{x,y\}\in E(\Omega,\overline\Omega)}\mu_{xy}|f(x)-f(y)|}{\sum_{x\in\delta\Omega}m_x|f(x)-\overline{f}|}.\end{eqnarray}
\end{rem}

We will need the following discrete version of Co-area formula. For discrete Co-area formula, see e.g. \cite[Lemma 3.3]{Grigor'yan2009}.
\begin{lemma}\label{discretecoareaformula}
For any $f\in\R^{\overline{\Omega}}$, 
we have
$$\sum_{e=\{x,y\}\in E(\Omega,\overline\Omega)}\mu_{xy}|f(x)-f(y)|=\int_{-\infty}^{\infty}\sum_{e=\{x,y\}\in E(\Omega,\overline\Omega),f(x)<\sigma\leq f(y)}\mu_{xy}d\sigma.$$
\end{lemma}
\begin{proof}
For any interval $(a,b]$, we denote by $\chi_{(a,b]}$ the characteristic function on $(a,b]$, i.e.
\[
\chi_{(a,b]}(x)=\left\{
       \begin{array}{ll}
        0,& x\notin (a,b], \\
        1,&x\in (a,b].
               \end{array}
     \right.
\]

\begin{eqnarray*}
\int_{-\infty}^{\infty}\sum_{e=\{x,y\}\in E(\Omega,\overline\Omega),  f(x)<\sigma\leq f(y)}\mu_{xy}d\sigma\nonumber &=&\int_{-\infty}^{\infty}\sum_{e=\{x,y\}\in E(\Omega,\overline\Omega)}\mu_{xy}\chi_{(f(x),f(y)]}(\sigma)d\sigma\nonumber\\
&=&\sum_{e=\{x,y\}\in E(\Omega,\overline\Omega)}\mu_{xy}\int_{-\infty}^{\infty}\chi_{(f(x),f(y)]}(\sigma)d\sigma\nonumber\\
&=&\sum_{e=\{x,y\}\in E(\Omega,\overline\Omega)}\mu_{xy}|f(x)-f(y)|.
\end{eqnarray*}
\end{proof}

Now we are ready to prove Theorem~\ref{sobolevconstant}.
\begin{proof}[Proof of Theorem~\ref{sobolevconstant}]
From Remark \ref{equivalence}, it suffices to prove \eqref{sobolevtypecharactrisation}. Choose $A\subset\overline\Omega$ that achieves $h_E(\widetilde{\Omega})$, i.e.
$$m(A\cap\delta\Omega)\leq m(A^{\vee}\cap\delta\Omega),\quad h_E(\widetilde{\Omega})=\frac{\mu(\partial_{\Omega} A)}{m(A\cap\delta\Omega)}.$$
Set
\[
u(x)=\left\{
       \begin{array}{ll}
        1,& x\in A, \\
        0,&x\in \overline{\Omega}\setminus A.
               \end{array}
     \right.
\]
We observe that  $\overline{u}$, the $L^1$-mean value of $u$, is contained in $[0,1]$. To see this, one just need to notice that if $m(A\cap\delta\Omega)=m(A^{\vee}\cap\delta\Omega)$, then $\overline{u}\in[0,1]$; if $m(A\cap\delta\Omega)<m(A^{\vee}\cap\delta\Omega)$, then $\overline{u}=0$.

For any $t\in[0,1]$,
\begin{eqnarray*}
\sum_{x\in\delta\Omega}m_x|u(x)-t|&=&\sum_{x\in A\cap\delta\Omega}m_x(1-t)+\sum_{x\in A^{\vee}\cap\delta\Omega}m_xt\nonumber\\
&=& m(A\cap\delta\Omega)+t\cdot\left(m(A^{\vee}\cap\delta\Omega)-m(A\cap\delta\Omega)\right)\nonumber\\
&\geq& m(A\cap\delta\Omega).
\end{eqnarray*}
Hence
\[
\frac{\sum_{e=\{x,y\}\in E(\Omega,\overline\Omega)}\mu_{xy}|u(x)-u(y)|}{\sum_{x\in\delta\Omega}m_x|u(x)-\overline{u}|} \leq \frac{\mu(\partial A\cap E(\Omega,\overline{\Omega}))}{m(A\cap\delta\Omega)}=h_E(\widetilde{\Omega}).
\]
Then it follows that
$$h_E(\widetilde{\Omega})\geq\inf_{f\in\R^{\overline\Omega}}\frac{\sum_{e=\{x,y\}\in E(\Omega,\overline\Omega)}\mu_{xy}|f(x)-f(y)|}{\sum_{x\in\delta\Omega}m_x|f(x)-\overline{f}|}.$$
Now we prove the opposite direction. For any nonconstant function $f\in\R^{\overline\Omega}$, choose a constant $c$ such that
$$m(\{f<c\}\cap\delta\Omega)\leq m(\{f\geq c\}\cap\delta\Omega),$$
$$m(\{f\leq c\}\cap\delta\Omega)\geq m(\{f> c\}\cap\delta\Omega).$$
Set $g:=f-c$, then we have
$$m(\{g<\sigma\}\cap\delta\Omega)\leq m(\{g\geq \sigma\}\cap\delta\Omega), \quad for\quad\sigma\leq0$$
and
$$m(\{g<\sigma\}\cap\delta\Omega)\geq m(\{g\geq \sigma\}\cap\delta\Omega), \quad for\quad\sigma>0.$$
For any $\sigma\in\R$, set
$$G(\sigma):=\sum_{e=\{x,y\}\in E(\Omega,\overline\Omega),g(x)<\sigma\leq g(y)}\mu_{xy}.$$
Then by Lemma \ref{discretecoareaformula}, we have
$$\sum_{e=\{x,y\}\in E(\Omega,\overline\Omega)}\mu_{xy}|f(x)-f(y)|=\int_{-\infty}^{\infty}G(\sigma)d\sigma.$$
Set $A:=\{g<\sigma\}$ for $\sigma\leq0$ and $A:=\{g\geq\sigma\}$ for $\sigma>0$. In either case, we have $m(A\cap\delta\Omega)\leq m(A^{\vee}\cap\delta\Omega)$ and
\[
G(\sigma)\geq h_E(\widetilde{\Omega})\cdot m(A\cap\delta\Omega)=h_E(\widetilde{\Omega})\cdot\left\{
       \begin{array}{ll}
        m(\{g<\sigma\}\cap\Omega),& \text{\ for}\quad \sigma\leq0, \\
        m(\{g\geq\sigma\}\cap\Omega),& \text{\ for}\quad \sigma>0
               \end{array}
     \right.
\]
by the definition of $h_E(\widetilde{\Omega})$. Hence
\begin{eqnarray*}
&&\sum_{e=\{x,y\}\in E(\Omega,\overline\Omega)}\mu_{xy}|f(x)-f(y)|=\int_{-\infty}^{0}G(\sigma)d\sigma+\int_{0}^{\infty}G(\sigma)d\sigma\nonumber\\
&\geq&h_E(\widetilde{\Omega})\left(\int_{-\infty}^{0}m(\{g<\sigma\}\cap\Omega)d\sigma+\int_{0}^{\infty} m(\{g\geq\sigma\}\cap\Omega)d\sigma\right)\nonumber\\
&=&h_E(\widetilde{\Omega})\left(\int_{-\infty}^{0}m(\{g<\sigma\}\cap\Omega)d\sigma+\int_{0}^{\infty} m(\{g\geq\sigma\}\cap\Omega)d\sigma\right)\nonumber\\
&=&h_E(\widetilde{\Omega})\left(\int_{-\infty}^{0}\sum_{x\in\delta\Omega}\chi_{(g(x),0]}m_xd\sigma+\int_{0}^{\infty}\sum_{x\in\delta\Omega}\chi_{(0,g(x)]}m_x d\sigma\right)\nonumber\\
&=&h_E(\widetilde{\Omega})\sum_{x\in\delta\Omega}|f(x)-c|m_x.
\end{eqnarray*}
Then the other direction follows and we complete the proof of the theorem.
\end{proof}

We obtain the following upper bound estimate for $\lambda_1(\Omega).$
\begin{prop}\label{upperbound}
Let $G$ be a finite graph and $\Omega\subset V,$ then we have
$$\lambda_1(\Omega)\leq 2h_E(\widetilde{\Omega}).$$
\end{prop}
\begin{proof}
Choose $A\subset\overline{\Omega}$ that achieves $h_E(\widetilde{\Omega})$, i.e. $m(A\cap\delta\Omega)\leq\frac12m(\delta\Omega)$ and $\frac{\mu(\partial_{\Omega} A)}{m(A\cap\delta\Omega)}=h_E(\widetilde{\Omega})$. Set $\varphi(x) \in\mathbb{R}^{\delta\Omega}$ as
\[
\varphi(x)=\left\{
       \begin{array}{ll}
        \frac{1}{m(A\cap\delta\Omega)},& x\in A\cap\delta\Omega, \\
        -\frac{1}{m(A^{\vee}\cap\delta\Omega)},& x\in A^{\vee}\cap\delta\Omega.
               \end{array}
     \right.
\]
Similarly set $\widetilde{\varphi}(x)$ to be
\[
\widetilde{\varphi}(x)=\left\{
       \begin{array}{ll}
        \frac{1}{m(A\cap\delta\Omega)},& x\in A, \\
        -\frac{1}{m(A^{\vee}\cap\delta\Omega)},& x\in A^{\vee}.
               \end{array}
     \right.
\]
Then we have
\begin{eqnarray*}
\lambda_1(\Omega)&\leq&\frac{D_{\Omega}(u_{\varphi})}{\sum_{x\in\delta\Omega}\varphi^2(x)m_x}
\leq\frac{D_{\Omega}(\widetilde{\varphi})}{\sum_{x\in\delta\Omega}\varphi^2(x)m_x}\nonumber\\
&=&\left(\frac{1}{m(A\cap\delta\Omega)}+\frac{1}{m(A^{\vee}\cap\delta\Omega)}\right)\mu(\partial A\cap E(\Omega,\overline{\Omega}))\nonumber\\
&\leq&2\frac{\mu(\partial_{\Omega} A)}{m(A\cap\delta\Omega)}=2h_E(\widetilde{\Omega}).
\end{eqnarray*}
The second inequality follows from the fact that harmonic functions minimize the Dirichlet energy among functions with the same boundary condition.
\end{proof}

\begin{figure}[!h]
\includegraphics[height=3.3cm,width=13cm]{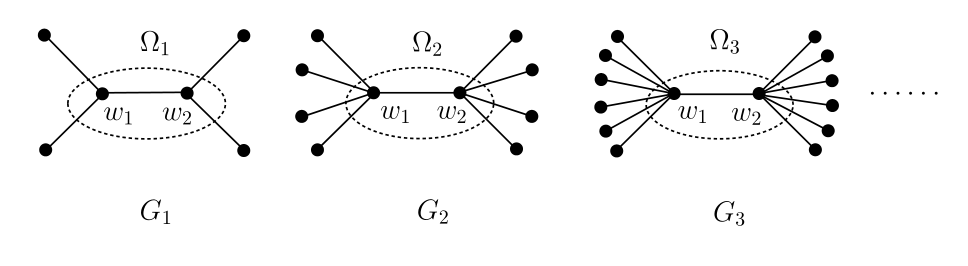}
\caption{}\label{figure2}
\end{figure}
\begin{rem}
The above upper bound estimate for $\lambda_1(\Omega)$ is sharp. From the following example, we can see that the factor 2 in the upper bound estimate can't be reduced.
\end{rem}

\begin{example}\label{exampleforupperbound}
Consider a sequence of graphs $\{G_n\}_{n=1}^{\infty}$ as shown in Figure 2 with $\Omega_n=\{w_1,w_2\}$, $\delta\Omega_n=\{v_1,v_2,\cdots,v_{4n}\}$ and unit edge weights. By calculation, $\lambda_1(G_n)=\frac{1}{n+1}$ and $h_E(G_n)=\frac{1}{2n}$. Hence
$\lambda_1(G_n)=\frac{1}{n+1}\leq\frac{1}{n}=2h_E(G_n).$
\end{example}

At the end of this section, we give the proof of Theorem~\ref{escobarlowerbound}.
\begin{proof}[Proof of Theorem~\ref{escobarlowerbound}]
Let $u$ be the first eigenvector associated to $\lambda_1(\Omega)$. For simplicity, we still denote by $u$ the harmonic extension of $u$. Set $v=(u-\overline{u})_+$ and choose $\overline{v}=0$. Then
$$m(\{v=0\}\cap\delta\Omega)\geq\frac12m(\delta\Omega).$$
Applying Theorem \ref{sobolevconstant} by choosing $f=v^2$, we have
\begin{eqnarray*}h_E(\widetilde{\Omega})\cdot\sum_{x\in\delta\Omega}v^2(x)m_x\leq\sum_{e=\{x,y\}\in E(\Omega,\overline{\Omega})}\mu_{xy}|v^2(x)-v^2(y)|.\end{eqnarray*}
For the right hand side of the above inequality,
\begin{eqnarray*}
&&\left(\sum_{e=\{x,y\}\in E(\Omega,\overline{\Omega})}\mu_{xy}\left|v^2(x)-v^2(y)\right|\right)^2\nonumber\\
&\leq&\sum_{e=\{x,y\}\in E(\Omega,\overline{\Omega})}\mu_{xy}(v(x)+v(y))^2\cdot\sum_{e=\{x,y\}\in E(\Omega,\overline{\Omega})}\mu_{xy}(v(x)-v(y))^2\nonumber\\
&\leq&2\sum_{e=\{x,y\}\in E(\Omega,\overline{\Omega})}\mu_{xy}(v^2(x)+v^2(y))\cdot\sum_{e=\{x,y\}\in E(\Omega,\overline{\Omega})}\mu_{xy}(v(x)-v(y))^2.\nonumber\\
\end{eqnarray*}
Notice that it suffices to consider the graph $\widetilde{\Omega}$, then
\begin{eqnarray*}
&&\sum_{e=\{x,y\}\in E(\Omega,\overline{\Omega})}\mu_{xy}(v^2(x)+v^2(y))\nonumber\\
&=&\frac{1}{2}\sum_{x,y\in\Omega}\mu_{xy}(v^2(x)+v^2(y))+\sum_{x\in\Omega}\sum_{y\in\delta\Omega}\mu_{xy}(v^2(x)+v^2(y))\nonumber\\
&=&\sum_{x\in\Omega}v^2(x)\left(\sum_{y\in\Omega}+\sum_{y\in\delta\Omega}\right)\mu_{xy}+\sum_{y\in\delta\Omega}v^2(y)\sum_{x\in\Omega}\mu_{xy}\nonumber\\
&=&\sum_{x\in\overline\Omega}v^2(x)m_x.
\end{eqnarray*}
Hence
\[
h_E(\widetilde{\Omega})\cdot\sum_{x\in\delta\Omega}v^2(x)m_x
\leq \left(2\sum_{x\in\overline\Omega}v^2(x)m_x\cdot \!\! \!\! \sum_{e=\{x,y\}\in E(\Omega,\overline{\Omega})}\!\! \mu_{xy}(v(x)-v(y))^2\right)^{1/2},
\]
 i.e.
\[ h_E(\widetilde{\Omega})\cdot\sum_{x\in\delta\Omega}(u-\overline{u})^2_+m_x
\leq \left(2\sum_{x\in\overline\Omega}(u-\overline{u})^2_+m_x\cdot\sum_{e=\{x,y\}\in E(\Omega,\overline{\Omega})}\mu_{xy}(u(x)-u(y))^2\right)^{1/2}.
\]
Similarly, we have
\[
h_E(\widetilde{\Omega})\cdot\sum_{x\in\delta\Omega}(u-\overline{u})^2_-m_x\nonumber
\leq \left(2\sum_{x\in\overline\Omega}(u-\overline{u})^2_-m_x\cdot\sum_{e=\{x,y\}\in E(\Omega,\overline{\Omega})}\mu_{xy}(u(x)-u(y))^2\right)^{1/2}.
\]
Hence
\begin{eqnarray}{\label{1}}
&&h_E(\widetilde{\Omega})\cdot\sum_{x\in\delta\Omega}(u-\overline{u})^2m_x\nonumber\\
&\leq&\left(2\sum_{x\in\overline\Omega}(u-\overline{u})^2m_x\cdot\sum_{e=\{x,y\}\in E(\Omega,\overline{\Omega})}\mu_{xy}(u(x)-u(y))^2\right)^{1/2}\nonumber\\
&\leq&\frac{a}{2}\sum_{x\in\overline\Omega}(u-\overline{u})^2m_x+\frac{1}{a}\sum_{e=\{x,y\}\in E(\Omega,\overline{\Omega})}\mu_{xy}(u(x)-u(y))^2\nonumber\\
&=&\frac{a}{2}\left(\sum_{x\in\Omega}+\sum_{x\in\delta\Omega}\right)(u-\overline{u})^2m_x\nonumber +\frac{1}{a}\sum_{e=\{x,y\}\in E(\Omega,\overline{\Omega})}\mu_{xy}(u(x)-u(y))^2.\\
\end{eqnarray}
For any $\varphi\in\mathds{R}^{\overline\Omega}$ we have
$$\mu_1(k)\leq\frac{D_{\Omega}(\varphi)+k\sum_{x\in\delta\Omega}\varphi^2(x)m_x}{\sum_{x\in\Omega}\varphi^2(x)m_x}.$$
Hence
$$\sum_{x\in\Omega}(u-\overline{u})^2m_x\leq\frac{D_{\Omega}(u)}{\mu_1(k)}+\frac{k}{\mu_1(k)}\sum_{x\in\delta\Omega}
(u-\overline{u})^2m_x.$$
Combining with the above inequality and \eqref{1}, we have
$$\left(h_E(\widetilde{\Omega})-\frac{a(k+\mu_1(k))}{2\mu_1(k)}\right)\frac{2a\mu_1(k)}{a^2+2\mu_1(k)}
\leq\frac{D_{\Omega}(u)}{\sum_{x\in\delta\Omega}(u-\overline{u})^2m_x}.$$
Using the fact that $u$ is the first eigenfunction associated to $\lambda_1$ we find that
$$\left(h_E(\widetilde{\Omega})-\frac{a(k+\mu_1(k))}{2\mu_1(k)}\right)\frac{2a\mu_1(k)}{a^2+2\mu_1(k)}
\leq\frac{\lambda_1\sum_{x\in\delta\Omega}u^2(x)m_x}{\sum_{x\in\delta\Omega}(u-\overline{u})^2m_x}.$$
Since $\langle u,1\rangle_{\delta\Omega}=0$, we have
$$\frac{\sum_{x\in\delta\Omega}u^2(x)m_x}{\sum_{x\in\delta\Omega}(u-\overline{u})^2m_x}\leq1$$
and therefore
\begin{eqnarray*}\lambda_1(\Omega)&\geq&\left(h_E(\widetilde{\Omega})-\frac{a(k+\mu_1(k))}{2\mu_1(k)}\right)\frac{2a\mu_1(k)}{a^2+2\mu_1(k)}\nonumber\\
&=&\frac{\left(2h_E(\widetilde{\Omega})\mu_1(k)-a(k+\mu_1(k))\right)a}{a^2+2\mu_1(k)}.
\end{eqnarray*}
\end{proof}

\begin{rem}
The maximum of the right hand side of \eqref{escobarcheegerestimate} with respect to $a$ can be achieved at
$$a_0=\frac{2h\mu}{\sqrt{(k+\mu)^2+2h^2\mu}+(k+\mu)}$$
and the maximum is
$$\frac{h^2\mu\sqrt{(k+\mu)^2+2h^2\mu}}{2h^2\mu+(k+\mu)^2+(k+\mu)\sqrt{(k+\mu)^2+2h^2\mu}}.$$
\end{rem}

\section{Jammes-type Cheeger estimate}
\begin{prop}
Let $G$ be a finite graph and $\Omega\subset V$, then we have
$$h_J(\widetilde{\Omega})\leq 1.$$
\begin{proof}
Choose $A=\{v\},$ where $v\in\delta\Omega$. Then by the definition of $h_J(\widetilde{\Omega})$, we have $h_J(\widetilde{\Omega})\leq\frac{\mu(\partial_{\Omega} A)}{m(A\cap\delta\Omega)}=\frac{m_v}{m_v}=1$.
\end{proof}
\end{prop}
The eigenvalues of $\Lambda$ can be characterised by Rayleigh quotient as follows
$$\lambda_1(\Omega)=\inf_{\varphi\in\R^{\delta\Omega},\|\varphi\|_{\ell^2}=1,\varphi\perp 1}\langle\Lambda\varphi,\varphi\rangle=\inf_{\varphi\in\R^{\delta\Omega},\|\varphi\|_{\ell^2}=1,\varphi\perp 1}D_{\Omega}(u_\varphi),$$
$$\lambda_{N-1}(\Omega)=\sup_{\varphi\in\R^{\delta\Omega},\|\varphi\|_{\ell^2}=1}\langle\Lambda\varphi,
\varphi\rangle=\sup_{\varphi\in\R^{\delta\Omega},\|\varphi\|_{\ell^2}=1}D_{\Omega}(u_\varphi).$$

We denote by $\sigma_1$ the first nontrivial eigenvalue of the Dirichlet-to-Neumann operator on a compact manifold with boundary. For convenience, we recall the idea of the proof of Jammes-type Cheeger estimate, which can be divided into four steps.

Step 1: Choosing $f$ as the eigenfunction associated to $\sigma_1$, we still denote by $f$ the harmonic extension of $f$ to $M$ with $\vol(M^+)\leq\frac{\vol(M)}{2}$, where $M^+:=\{x\in M|f(x)>0\}$.

Step 2: Show that $\sigma_1=\frac{\int_{M^+}|df|^2}{\int_{\partial M^+}f^2}$, where $\partial M^+=M^+\cap\partial M.$

Step 3: By H\"{o}lder's inequality,  $\sigma_1=\frac{(\int_{M^+}f^2)(\int_{M^+}|df|^2)}{(\int_{M^+}f^2)(\int_{\partial M^+}f^2)}\geq\frac{1}{4}\frac{(\int_{M^+}|d(f^2)|)^2}{(\int_{M^+}f^2)(\int_{\partial M^+}f^2)}$.

Step 4: Set $D_t:=f^{-1}([\sqrt{t},\infty))$, $\partial_ID_t:=\partial D_t\cap \mathrm{int}(M)$ and  $\partial_ED_t:=\partial D_t\setminus\partial_ID_t$. Use Co-area formula to show that $\int_{M^+}|d(f^2)|=\int_{t\geq 0}\mathrm{Area}(\partial_ID_t)$, $\int_{M^+}f^2=\int_{t\geq 0}\vol(D_t)$ and $\int_{\partial M^+}f^2=\int_{t\geq0}\mathrm{Area}(\partial_ED_t)$.

Then by the definitions of $h_M$ and $h_J(M)$ the lower bound estimate of $\sigma_1$ follows.

Inspired by the Riemannian case, we can prove the Jammes-type Cheeger constant for $\lambda_1(\Omega)$ in the discrete setting. Let $0\neq f\in\R^{\delta\Omega}$ be an eigenfunction associated to $\lambda_1(\Omega)$. For convenience, we still denote $u_{f}$ by $f(x)$. Set $\overline{\Omega}^{+}=\{x\in\overline\Omega\mid f(x)> 0\}$  with $m(\overline{\Omega}^+)\leq\frac{m(\overline\Omega)}{2}$ (otherwise we can change the sign of $f$) and set
\begin{equation}\label{positivepart}
g(x)=\left\{
       \begin{array}{ll}
        f(x),& x\in \overline{\Omega}^{+}, \\
        0,&otherwise.
               \end{array}
     \right.
\end{equation}
For simplicity, we set $\overline{\Omega}^-:=\overline{\Omega}\setminus\overline{\Omega}^+$, $\Omega^+:=\overline{\Omega}^{+}\cap\Omega$, $\Omega^-:=\Omega\setminus\Omega^+$, $\delta^{+}\Omega:=\overline{\Omega}^{+}\cap\delta\Omega$ and $\delta^{-}\Omega:=\delta\Omega\setminus\delta^{+}\Omega$. In order to prove Jammes-type Cheeger estimate, we need the following Lemmas.

\begin{lemma}\label{lemma1}
For $g$ as in \eqref{positivepart}, we have
\begin{eqnarray}\label{eq:eq3}
\lambda_1(\Omega)\geq\frac{\sum_{e=\{x,y\}\in E(\overline{\Omega}^+,\overline{\Omega})}\mu_{xy}(g(y)-g(x))^{2}}{\sum_{x\in\delta^{+}\Omega}g^2(x)m_{x}}.
\end{eqnarray}
\end{lemma}

\begin{proof}
Notice that it suffices to consider graph $\widetilde{\Omega}$. Hence for any $x\in\delta\Omega$, we have
$$\Delta f(x)=\frac{1}{m_x}\sum_{y\in\Omega}\mu_{xy}(f(y)-f(x))=-\frac{\partial f(x)}{\partial n}.$$ Then
\begin{eqnarray}\label{eigenvaluelowerestimate}
\langle\Delta f(x),g(x)\rangle_{\overline{\Omega}^+}&=&\sum_{x\in\Omega^+}\Delta f(x)g(x)m_x+\sum_{x\in\delta^+\Omega}\Delta f(x)g(x)m_x\nonumber\\
&=&-\lambda_1(\Omega)\sum_{x\in\delta^+\Omega}g^2(x)m_x.
\end{eqnarray}
Notice that
\begin{eqnarray*}
&&\langle\Delta f(x),g(x)\rangle_{\overline{\Omega}^+}=\left(\sum_{x\in\overline{\Omega}^+}\sum_{y\in\overline{\Omega}^+}+
\sum_{x\in\overline{\Omega}^+}\sum_{y\in\overline{\Omega}^-}\right)\mu_{xy}(f(y)-f(x))g(x)\nonumber\\
&=&-\frac12\sum_{x,y\in\overline{\Omega}^+}\mu_{xy}(f(y)-f(x))(g(y)-g(x))-\sum_{x\in\overline{\Omega}^+}
\sum_{y\in\overline{\Omega}^-}\mu_{xy}(f(x)-f(y))(g(x)-g(y))\nonumber\\
&\leq&-\sum_{e=\{x,y\}\in E(\overline{\Omega}^+,\overline{\Omega}^+)}\mu_{xy}(g(y)-g(x))^2-\sum_{x\in\overline{\Omega}^+}
\sum_{y\in\overline{\Omega}^-}\mu_{xy}(g(x)-g(y))^2\nonumber\\
&=&-\sum_{e=\{x,y\}\in E(\overline{\Omega}^+,\overline{\Omega})}\mu_{xy}(g(y)-g(x))^{2}.
\end{eqnarray*}
Then the lemma follows in view of \eqref{eigenvaluelowerestimate}.
\end{proof}
Multiplying both the numerator and denominator of the fraction in the right hand side of \eqref{eq:eq3} by $\sum_{x\in\overline{\Omega}^+}g^2(x)m_x$ and setting
$$\frac{P}{Q}:=\frac{\sum_{x\in\overline{\Omega}^+}g^2(x)m_x\cdot \sum_{e=\{x,y\}\in E(\overline{\Omega}^+,\overline{\Omega})}\mu_{xy}(g(y)-g(x))^{2}}{\sum_{x\in\overline{\Omega}^+}g^2(x)m_x\cdot\sum_{x\in
\delta^{+}\Omega}g^2(x)m_{x}},$$
we have
\begin{eqnarray}
\lambda_1(\Omega)\geq\frac{P}{Q}.
\end{eqnarray}

\begin{lemma}\label{lemma2}
$$P\geq\frac12\left(\sum_{e=\{x,y\}\in E(\overline{\Omega}^+,\overline{\Omega})}|g^2(x)-g^2(y)|\mu_{xy}\right)^2.$$
\end{lemma}
\begin{proof}
Note that
\begin{eqnarray*}
&&\sum_{x\in\overline{\Omega}^{+}}g^2(x)m_x=\left(\sum_{x\in\overline{\Omega}^{+}}\sum_{y\in\overline{\Omega}^+}+\sum_{x\in
\overline{\Omega}^{+}}\sum_{y\in\overline{\Omega}^-}\right)g^2(x)\mu_{xy}\nonumber\\
&=&\frac12\sum_{x,y\in\overline{\Omega}^+}(g^2(x)+g^2(y))\mu_{xy}+\sum_{x\in
\overline{\Omega}^{+}}\sum_{y\in\overline{\Omega}^-}g^2(x)\mu_{xy}\nonumber\\
&=&\sum_{e=\{x,y\}\in E(\overline{\Omega}^+,\overline{\Omega}^+)}(g^2(x)+g^2(y))\mu_{xy}+\sum_{x\in
\overline{\Omega}^{+}}\sum_{y\in\overline{\Omega}^-}g^2(x)\mu_{xy}\nonumber\\
&=&\sum_{e=\{x,y\}\in E(\overline{\Omega}^+,\overline{\Omega})}(g^2(x)+g^2(y))\mu_{xy}.
\end{eqnarray*}
Hence
\begin{eqnarray*}
P&=&\sum_{e=\{x,y\}\in E(\overline{\Omega}^+,\overline{\Omega})}(g^2(x)+g^2(y))\mu_{xy}\cdot \sum_{e=\{x,y\}\in E(\overline{\Omega}^+,\overline{\Omega})}\mu_{xy}(g(y)-g(x))^{2}\nonumber\\
&\geq&\frac12\sum_{e=\{x,y\}\in E(\overline{\Omega}^+,\overline{\Omega})}(g(x)+g(y))^2\mu_{xy}\cdot\sum_{e=\{x,y\}\in E(\overline{\Omega}^+,\overline{\Omega})}\mu_{xy}(g(y)-g(x))^{2}\nonumber\\
&\geq&\frac12\left(\sum_{e=\{x,y\}\in E(\overline{\Omega}^+,\overline{\Omega})}|g^2(x)-g^2(y)|\mu_{xy}\right)^2.
\end{eqnarray*}
The last inequality follows from H\"{o}lder's inequality.
\end{proof}
For any $t>0$, set $D_t:=g^{-1}([\sqrt{t},+\infty))=\{x\in\overline\Omega|g^2(x)\geq t\}$. Then we have $m(D_t)\leq m(\overline{\Omega}^+)\leq\frac{m(\overline\Omega)}{2}$.
\begin{lemma}
$$\int_{0}^{\infty}\mu(\partial D_t\cap E(\Omega,\overline\Omega))dt=\sum_{e=\{x,y\}\in E(\overline{\Omega}^+,\overline{\Omega})}\mu_{xy}|g^2(x)-g^2(y)|.$$
\end{lemma}
\begin{proof}
It follows from Lemma \ref{discretecoareaformula} by setting $f=g^2$ and considering the edge set $E(\overline{\Omega}^+,\overline{\Omega})$.
\end{proof}

\begin{lemma}\label{lemma4}
$$\int_{0}^{\infty}m(D_t)dt=\sum_{x\in\overline{\Omega}^+}g^2(x)m_x .$$
$$\int_{0}^{\infty}m(D_t\cap\delta\Omega)dt=\sum_{x\in\delta^+\Omega}g^2(x)m_x.$$
\end{lemma}
\begin{proof}
Similar to Lemma \ref{discretecoareaformula}, we have
\[\int_{0}^{\infty}m(D_t)dt = \int_{0}^{\infty}\sum_{x\in D_t}m_xdt = \int_{0}^{\infty}\sum_{x\in\overline{\Omega}^+}m_x\chi_{(0,g^2(x)]}(t)dt = \sum_{x\in\overline{\Omega}^+}g^2(x)m_x
\]
and
\[
\int_{0}^{\infty}m(D_t\cap\delta\Omega)dt = \int_{0}^{\infty}\sum_{x\in\delta^+\Omega}m_x\chi_{(0,g^2(x)]}(t)dt
 =\sum_{x\in\delta^+\Omega}g^2(x)m_x.
\]
\end{proof}

Now we are ready to prove Theorem \ref{thm:main3}.
\begin{proof}[Proof of Theorem~\ref{thm:main3}]

Combining with the above Lemmas (Lemma \ref{lemma1} to \ref{lemma4}), we have
\begin{eqnarray*}
\lambda_1(\Omega)&\geq&\frac{1}{2}\frac{\int_{0}^{\infty}\mu(\partial D_t\cap E(\Omega,\overline\Omega))dt\cdot\int_{0}^{\infty}\mu(\partial D_t\cap E(\Omega,\overline\Omega))dt}{\sum_{x\in\overline{\Omega}^+}g^2(x)m_x\cdot\sum_{x\in
\delta^{+}\Omega}g^2(x)m_{x}}\nonumber\\
&\geq&\frac{1}{2}\frac{\int_{0}^{\infty}h(\widetilde{\Omega})m(D_t)dt\cdot\int_{0}^{\infty}h_J(\widetilde{\Omega})
m(D_t\cap\delta\Omega)dt}{\sum_{x\in\overline{\Omega}^+}g^2(x)m_x\cdot\sum_{x\in
\delta^{+}\Omega}g^2(x)m_{x}}\nonumber\\
&=&\frac{h(\widetilde{\Omega})h_J(\widetilde{\Omega})}{2}\frac{\int_{0}^{\infty}m(D_t)dt\cdot\int_{0}^{\infty}m( D_t\cap\delta\Omega)dt}{\sum_{x\in\overline{\Omega}^+}g^2(x)m_x\cdot\sum_{x\in
\delta^{+}\Omega}g^2(x)m_{x}}\nonumber\\
&=&\frac{h(\widetilde{\Omega})h_J(\widetilde{\Omega})}{2}.
\end{eqnarray*}
\end{proof}

From Theorem \ref{thm:main3}, we know that $h(\widetilde{\Omega})$ plays an important role in Jammes-type Cheeger estimate. Recall that $\zeta_1(\widetilde{\Omega})$ is the first nontrivial eigenvalue of the Laplace operator with no boundary condition on $\widetilde\Omega$ and the classical Cheeger estimates reads as
\begin{eqnarray}\label{finitecheegerestimate}
2h(\widetilde{\Omega})\geq\zeta_1(\widetilde{\Omega})\geq\frac{h(\widetilde{\Omega})^2}{2},
\end{eqnarray}
see \cite[p.26]{Chung97}. Then we are ready to prove Corollary \ref{application}
\begin{proof}[Proof of Corollary~\ref{application}]
Recall that $h(\widetilde{\Omega})\leq h_J(\widetilde{\Omega})$. Combining with Theorem \ref{thm:main3} and \eqref{finitecheegerestimate}, we have
\begin{eqnarray}\label{eigenvaluecomparesion}\lambda_1(\Omega)\geq \frac{h(\widetilde{\Omega})^2}{2}\geq\frac{(\zeta_{1}(\widetilde{\Omega}))^2}{8}.\end{eqnarray}
\end{proof}

Finally we give an example to show that $\lambda_1$ can't be bounded from below using only $h_J(\widetilde{\Omega})$.

\begin{example}\label{exapmle}
Consider a path graph with even vertices n ($n\geq 6$) and unit edge weights as shown in Figure 1. Set $\Omega=\{v_2,v_3,\cdots,v_{n-1}\}$, $\delta_E\Omega=\{v_1,v_n\}$. Then we have

\begin{equation*}
\Lambda=\left(\begin{array}{cc}
\frac{1}{n-1} & -\frac{1}{n-1} \\
-\frac{1}{n-1} & \frac{1}{n-1}
\end{array}
\right).
\end{equation*}
Hence $\lambda_0(\Omega)=0$ and $\lambda_1(\Omega)=\frac{2}{n-1}$. Choosing $A=\{v_1,v_2,\cdots,v_{\frac{n}{2}}\}$, we obtain that  $h(\widetilde{\Omega})=\frac{1}{n-1}$ and $h_J(\widetilde{\Omega})=1$. Hence the Jammes-type Cheeger estimate we obtained is asymptotically sharp of the same order $\frac{1}{n-1}$ on both sides as $n\rightarrow\infty$. Moreover, one can not obtain that $\lambda_1(\Omega)\geq F(h_J(\widetilde{\Omega}))$ for any positive function $F$.

\end{example}

\bibliography{DTNproblems}

\newcommand{\etalchar}[1]{$^{#1}$}
\begin{thebibliography}{KKK{\etalchar{+}}14}

\bibitem[Ban80]{Bandle1980}
C.~Bandle.
\newblock {\em Isoperimetric inequalities and applications}, volume~7 of {\em
  Monographs and Studies in Mathematics}.
\newblock Pitman (Advanced Publishing Program), Boston, Mass.-London, 1980.

\bibitem[BH12]{BrouwerHaemers2012}
A.E. Brouwer and W.H. Haemers.
\newblock {\em Spectra of graphs}.
\newblock Universitext. Springer, New York, 2012.

\bibitem[Cal80]{Calderon1980}
A.P. Calder{\'o}n.
\newblock On an inverse boundary value problem.
\newblock {\em Seminar on Numerical Analysis and its Applications to Continuum
  Physics, Rio de Janeiro, Editors W.H. Meyer and M.A. Raupp, Sociedade
  Brasileira de Matematica}, pages 65--73, 1980.

\bibitem[Cha16]{Chang2014}
K.~C. Chang.
\newblock Spectrum of the 1-{L}aplacian and {C}heeger's constant on graphs.
\newblock {\em J. Graph Theory}, 81(2):167--207, 2016.

\bibitem[Che70]{Cheeger1970}
J.~Cheeger.
\newblock A lower bound for the smallest eigenvalue of the {L}aplacian.
\newblock In {\em Problems in analysis ({P}apers dedicated to {S}alomon
  {B}ochner, 1969)}, pages 195--199. Princeton Univ. Press, Princeton, NJ,
  1970.

\bibitem[Chu97]{Chung97}
F.R.K. Chung.
\newblock {\em Spectral Graph Theory}, volume~92 of {\em CBMS Regional
  Conference Series in Mathematics}.
\newblock Published for the Conference Board of the Mathematical Sciences,
  Washington, DC; by the American Mathematical Society, Providence, RI, 1997.

\bibitem[CSZ15]{ChangShaoZhang2015}
K.~C. Chang, Sihong Shao, and Dong Zhang.
\newblock The 1-{L}aplacian {C}heeger cut: theory and algorithms.
\newblock {\em J. Comput. Math.}, 33(5):443--467, 2015.

\bibitem[Esc97]{Escobar1997}
J.~Escobar.
\newblock The geometry of the first non-zero {S}tekloff eigenvalue.
\newblock {\em Journal of Functional Analysis}, 150(2):544--556, 1997.

\bibitem[Gri09]{Grigor'yan2009}
A.~Grigor'yan.
\newblock Analysis on {g}raphs.
\newblock {\em Lecture {N}otes, {U}niversity of {B}ielefeld}, 2009.
\newblock \url{https://www.math.uni-bielefeld.de/~grigor/aglect.pdf}.

\bibitem[Jam15]{Jammes2015}
P.~Jammes.
\newblock Une in\'egalit\'e de {C}heeger pour le spectre de {S}teklov.
\newblock {\em Annales de l'Institut Fourier}, 65(3):1381--1385, 2015.

\bibitem[KKK{\etalchar{+}}14]{Kuznetsov2014}
N.~Kuznetsov, T.~Kulczycki, M.~Kwasnicki, A.~Nazarov, B.~Siudeja, S.~Poborchi,
  and I.~Polterovich.
\newblock The legacy of {V}ladimir {A}ndreevich {S}teklov.
\newblock {\em Notices of the American Mathematical Society}, 61(1):p¨¢gs.
  9--23, 2014.

\bibitem[Law10]{Lawler2010}
Gregory~F. Lawler.
\newblock {\em Random walk and the heat equation}, volume~55 of {\em Student
  Mathematical Library}.
\newblock American Mathematical Society, Providence, RI, 2010.

\bibitem[Li12]{GeometryAnalysis2012}
P.~Li.
\newblock {\em Geometric analysis}, volume 134 of {\em Cambridge Studies in
  Advanced Mathematics}.
\newblock Cambridge University Press, Cambridge, 2012.

\bibitem[LL10]{LawlerLimic2010}
Gregory~F. Lawler and V.~Limic.
\newblock {\em Random walk: a modern introduction}, volume 123 of {\em
  Cambridge Studies in Advanced Mathematics}.
\newblock Cambridge University Press, Cambridge, 2010.

\bibitem[Tay96]{Taylor1996}
M.E. Taylor.
\newblock {\em Partial differential equations. {II}. Qualitative studies of
  linear equations}, volume 116 of {\em Applied Mathematical Sciences}.
\newblock Springer-Verlag, New York, 1996.

\bibitem[Uhl14]{Uhlmann2014}
G.~Uhlmann.
\newblock Inverse problems: seeing the unseen.
\newblock {\em Bull. Math. Sci.}, 4(2):209--279, 2014.

\end{thebibliography}
\bibliographystyle{alpha}

\end{document}